\documentclass[11pt]{article}

\usepackage{amsfonts, amsmath, amssymb, amsthm}
\usepackage[colorlinks=true,citecolor=blue]{hyperref}
\usepackage[margin = 2.50cm]{geometry}
\usepackage{mathtools}
\usepackage{graphicx}
\usepackage{enumerate}
\usepackage{verbatim}
\usepackage{tikz}
\usepackage{algorithmic}
\usepackage[utf8]{inputenc}
\usepackage{microtype}
\usepackage{amsmath,amsthm,amssymb,color,tikz,mathtools}

\usepackage{tcolorbox}
\usepackage{mdframed}

\usetikzlibrary{positioning}
\usetikzlibrary{shapes}
\usetikzlibrary{decorations.pathreplacing}
\usepackage{enumerate, enumitem}

\newcommand{\FF}{\mathbb{F}}

\newcommand{\NN}{\mathbb{N}}

\newcommand{\RR}{\mathbb{R}}

\newcommand{\ZZ}{\mathbb{Z}}

\newcommand{\cQ}{\mathcal{Q}}

\newcommand{\set}[1]{\left\{ #1 \right\}}

\newcommand{\zone}{\set{0,1}}

\renewcommand{\epsilon}{\varepsilon}

\renewcommand{\deg}{\text{deg}}

\newtheorem{theorem}{Theorem}
\newtheorem{corollary}[theorem]{Corollary}
\newtheorem{remark}[theorem]{Remark}

\newtheorem{defi}[theorem]{Definition}

\newtheorem{proposition}[theorem]{Proposition}

\newcommand{\remove}[1]{}
\newcommand{\ic}{\mathrm{ic}}
\newcommand{\w}{\mathrm{wt}}
\newcommand{\spc}{\mathrm{spc}}

\begin{document}

\title{
About almost covering subsets of the hypercube
}

\author{
Arijit Ghosh \footnote{Indian Statistical Institute, Kolkata, India}
\and
Chandrima Kayal \footnote{The Institute of Mathematical Science, Chennai, India}
\and 
Soumi Nandi 
\footnote{Indian Institute of Science, Bengaluru India}
}

\maketitle

\begin{abstract}
Let $\FF$ be a field and consider the hypercube $\{0,1\}^{n}$ in $\FF^{n}$. Sziklai and Weiner (Journal of Combinatorial Theory, Series A 2022) showed that if a polynomial $P(X_{1}, \dots, X_{n}) \in \FF [X_{1}, \dots, X_{n}]$ vanishes on every point of the hypercube $\{0,1\}^{n}$ except those with at most $r$ many ones then the degree of the polynomial will be at least $n-r$. This is a generalization of Alon and F\"uredi's fundamental result (European Journal of Combinatorics 1993) about polynomials vanishing on every point of the hypercube except at the origin (point with all zero coordinates). 
Sziklai and Weiner proved their result using M\"{o}bius inversion formula and the Zeilberger method for proving binomial equalities. In this short note, we show that a stronger version of Sziklai and Weiner's result can be derived directly from Alon and F\"uredi's result. We also prove a multiplicity version of our results when $\FF = \RR$.
\end{abstract}

\section{Introduction}

    Let $\FF$ be a field (finite or infinite) and  $\zone^{n}$ is the $n$ dimensional hypercube embedded in $\FF^{n}$. Also, let $\FF [ X_1, \dots, X_n]$ denotes the polynomial ring over the field $\FF$.
     Alon and F\"{u}redi,  proved the following influential result about the degrees of polynomials vanishing on the subsets of the hypercube using Combinatorial Nullstellensatz Theorem~\cite{AT92,Alon99}. Over the years both Combinatorial Nullstellensatz Theorem, and Alon and Furedi’s result have found multiple extensions and applications. The formal statement of the result is as follows:
     
\begin{theorem}[Alon and F\"{u}redi~\cite{AF93}]
   Suppose $P(X_{1}, \dots, X_{n})$ is a polynomial in $\FF \left[ X_{1}, \dots, X_{n} \right]$ such that  $P(u)=0$ for all $u\in\zone^n\setminus\{(0,\dots,0)\}$ and $P(0,\dots,0)\neq 0$. Then $deg(P) \geq n$. 
   \label{th: AF}
\end{theorem}
\noindent

This result has inspired various generalizations and found applications in multiple areas. In this work, we will focus on one such generalization of Theorem~\ref{th: AF} over the field $\FF$. Specifically, for any subset $V \subset \zone^n$ what is the minimum degree of the polynomial $P(X_1, \dots, X_n) \in \FF [X_1, \dots, X_n]$ that vanishes over all points in $V$ and is non-zero at least at one point  $u \in \zone^n \setminus V$? 

For any $u=(u_1,\dots,u_n)\in\zone^n$, by \emph{weight} of $u$ we mean the $L_1$-norm $\|u\|_{1}$ of $u$, that is, the number of $1$'s in the coordinates of $u$. 
Recently, using {\em M\"{o}bius inversion formula}~\cite[Chapter~3]{Stanley_2011} and Zeilberger's algorithm for proving binomial coefficient identities~\cite{PauleS95}, Sziklai and Weiner proved an interesting extension of Alon and F\"{u}redi's result (Theorem~\ref{th: AF}). 
\begin{theorem}[Sziklai and Weiner~\cite{SziklaiW23}]
    Suppose $P(X_{1}, \dots, X_{n})$ is a polynomial in $\FF[X_1,\dots,X_n]$ such that $P(u) \neq 0$ for all $u \in \zone^n$ with $\| u \|_{1} \leq r$, and $P(v) = 0$
    for all $v\in\zone^n$ with $\|v \|_{1} >r$. Then $\deg(P)\geq n-r$.
    \label{SW23}
\end{theorem}
\noindent
Using Gr\"{o}bner basis theory, Heged\"{u}s gave a further generalization of Theorem~\ref{SW23}.

\begin{theorem}[Heged\"{u}s~\cite{hegedus2024coversfinitesetspoints}]
Suppose $P(X_{1},\dots,X_{n})$ is a polynomial in $\FF[X_1,\dots,X_n]$ such that $P(u)=0$ for all $u\in\zone^n$ with $\| u \|_{1} \leq k$, and there exists $v\in\zone^n$ such that $\| v \|_{1} > k$ and $P(v)\neq 0$. Then $\deg(P)>k$.
\label{thm-Hegedus}
\end{theorem}
\noindent
To see why Theorem~\ref{thm-Hegedus} is a generalization of Theorem~\ref{SW23}, suppose $R$ is a polynomial in $\FF[X_1,\dots,X_n]$ such that for any $u\in\zone^n,\;R(u)=0$ if and only if $\| u \|_{1} > r$.
Now consider the polynomial $P(X_1,\dots,X_n)=R(1-X_1,\dots,1-X_n)$. \remove{Here $-1$ is the additive inverse of $1$ in the field $\FF$.} Observe that $\deg(P)=\deg(R)$ and $P(u)=0$ except for all $u\in\zone^n$ with $\| u \|_{1} > n-r-1$. Therefore, from Theorem~\ref{thm-Hegedus}, we have $\deg(P)\geq n-r$.

The main objective of this short note is to show that 
the following generalization of Theorem~\ref{SW23} and Theorem~\ref{thm-Hegedus} directly follows from Alon and F\"{u}redi's result (Theorem~\ref{th: AF}). 

\begin{theorem}[Main result]
    Let $S$ be a proper subset of $\zone^{n}$. 
    Suppose $P\in\FF[X_1,\dots,X_n]$ be a polynomial with $P(u)=0$ for all $u \in S$ and there exists a $v\in \zone^{n} \setminus S$ with $P(v) \neq 0$. Then, 
    $$
        \deg(P)\geq \max\{w, n-W\},
    $$ 
    where $w = \min \left\{ \| u \|_{1}  :  P(u)\neq 0 \right\}$ and $W = \max \left\{ \| u \|_{1} : P(u)\neq 0 \right\}$.
    \label{thm-new-lower-bound}
\end{theorem}
\noindent
Observe that Theorem~\ref{thm-Hegedus} is a simple corollary of the above result, and we have already seen that Theorem~\ref{thm-Hegedus} is a 
generalization of Theorem~\ref{SW23}. Moreover, several recent results~\cite{GKN23,Venkitesh22, GKNV23} about 
degrees of real polynomials vanishings on subsets of $\{0,1\}^{n} \subseteq \RR^{n}$
can be obtained as corollaries of Theorem~\ref{thm-new-lower-bound}. Note that in circuit complexity 
one studies the degrees of polynomials in $\mathbb{Z}_{N}[X_{1}, \dots, X_{n}]$, where $\mathbb{Z}_{N} :=  \mathbb{Z}/N\mathbb{Z}$ and $N \in \NN$, that vanishes on specific subsets of the hypercube 
$\{0,1\}^{n}$ to prove  
explicit lower bounds for the algebraic degree of Boolean functions like OR~\cite{BBR94} and Majority~\cite{Szegedy89}, etc. In contrast, our result holds for any subset of $\zone^n$ considered to be the zero set of the polynomial. Furthermore, compared to the previous results~\cite{GKNV23}, our result also applies to non-symmetric cases. That is, we do not require the vanishing set of the polynomials to be symmetric. In this sense, Theorem~\ref{thm-new-lower-bound} can break the symmetry. 
Additionally when $\FF = \RR$, we can prove a multiplicity version of Theorem~\ref{thm-new-lower-bound}, see Theorem~\ref{th: mult_R}.


\section{Definitions and notations}\label{sec: def_not}

We will use the following definitions and notations in this paper. 
\begin{itemize}

    \item 
        We will denote by $\RR$ the set of all real numbers, $\NN$ the set of all natural numbers, and $\mathbb{Z}$ the set of all integers.

    \item 
        Given a set $S$, $\binom{S}{k}$ denotes the set of all $k$ sized subsets of $S$.
    
    \item 
        For any $n \in \NN$, $[n]$ denotes the set $\left\{ 1, \dots, n\right\}$.

    \item 
        Given $n \in \NN$, $\mathbb{Z}_{n}$ denotes the ring $\mathbb{Z}/{\left( n\mathbb{Z} \right)}$, and $\FF_{2}$ denotes the field with two elements. 

    \item 
        $\mathcal{R}[X_{1}, \dots, X_{n}]$ denotes the polynomial ring over the ring $\mathcal{R}$. 

    \item 
        $\mathcal{O}$ denotes the point in $ (0, \dots, 0)\in \left\{ 0,1 \right\}^{n}$, that is, all the coordinates of $\mathcal{O}$ are set to $0$.

    \item 
        Given a polynomial $P(X_{1}, \dots, X_{n})$ over the polynomial ring $\mathcal{R}[X_{1}, \dots, X_{n}]$, $deg(P)$ denotes the degree of the polynomial $P(X_{1}, \dots, X_{n})$.
    
    \item 
        Given $u \in \{ 0,1\}^{n}$, {\em weight} of $u$ denotes the $L_{1}$-norm $\|u\|_{1}$ of $u$, that is, number of $1$'s in the coordinates of $u$.

\end{itemize}

\section{Covering with polynomials from $\FF[X_1, \dots, X_n]$}

We now proceed to prove Theorem~\ref{thm-new-lower-bound}.

\begin{proof}[Proof of Theorem~\ref{thm-new-lower-bound}]
We can assume that $w \geq 1$, as $w = 0$ case directly follows from Theorem~\ref{th: AF}.
Without loss of generality we can also assume that $w\geq n-W$, otherwise, we work with the polynomial $\widetilde{P}(X_{1}, \dots, X_{n}) := P(1-X_{1}, \dots, 1-X_{n})$ in place of the polynomial $P$.
Let $u = (u_{1}, \dots, u_{n}) \in \zone^{n} \setminus S$ such that $\|u\|_{1} = w$ and $P(u)\neq 0$. Then, for all $v\in\zone^n$ with $\|v\|_{1} < w$, we have $P(v)=0$. Without loss of generality we may assume that $u_i=1$ if and only if $i\in[w]$. Consider the polynomial 
$$
    Q(X_1,\dots,X_w) : = P(1-X_1,\dots, 1-X_w,0,\dots,0). 
$$
Then $\deg(Q)\leq\deg(P)$ and $Q(0,\dots,0)=P(u)\neq 0$. 
Now take any $y\in\zone^w\setminus\{(0,\dots,0)\}$. Then $Q(y)=P(\Tilde{y})$, where $\Tilde{y}_{i}=1-y_i$ for all $i\in [w]$ and $\Tilde{y}_{i}=0$ for all $i\in [n]\setminus [w]$. Since $\w(\Tilde{y})<w$, we have $P(\Tilde{y})=0$. So the polynomial $Q(X_{1}, \dots, X_{w})$ vanishes everywhere on $\zone^w$, except at $(0,\dots,0)\in\zone^w$. Therefore by Theorem~\ref{th: AF}, $\deg(Q)\geq w$. Hence, $\deg(P)\geq w$.
\end{proof}


\begin{corollary}
    Suppose $P\in\FF_2[X_1,\dots,X_n]$ such that $P$ vanishes at $v\in\zone^n$ if and only if $\|w\|_{1}=k$.
    Then $\deg(P)\geq\max\{k,n-k\}$.
\end{corollary}

\begin{proof}
    Let $Q\in\FF_2[X_1,\dots,X_n]$ is defined by $Q(X_1,\dots,X_n)=1-P(X_1,\dots,X_n)$. Then $\deg(Q)=\deg(P)$ and $Q$ vanishes at each $u\in\zone^n$ except when $\w(u)=k$.
    Then by Theorem~\ref{thm-new-lower-bound}, $\deg(Q)\geq\max\{k,n-k\}$.  
\end{proof}

\begin{remark}
    The above result is tight.
\end{remark}

The following proposition establishes the

\begin{proposition}
    Let $q\in \NN$ be a prime, and $0\leq k \leq n$ satisfy the following: 
    \begin{itemize}
        \item $q^{m-1}\leq n-k<q^m$ and
        \item $k+1$ is divisible by $q^m$.
    \end{itemize}
    If $P(X_{1}, \dots, X_{n}) \in\FF_q[X_1,\dots,X_n]$ is a polynomial that vanishes at each $u\in\zone^n$ except when $\|u\|_{1}=k$ then $\deg(P)=\max\{k,n-k\}$.
\label{prop: tightness}\end{proposition}

We need the following result by Kummer~\cite{Kummer1852} to prove Proposition~\ref{prop: tightness}.

\begin{theorem}[Kummer~\cite{Kummer1852}]
    Let $j$ and $m$ be integers such that $0\leq j\leq m$. Let $\alpha\in\NN$ and $p$ be a prime. Then $p^\alpha$ divides $\binom{m}{j}$ if and only if $\alpha$ carries are needed when adding $j$ and $m-j$ in base $p$.
\label{kummer}\end{theorem}

\begin{proof}[Proof of Proposition~\ref{prop: tightness}]
    Without loss of generality we may assume that $k\geq\frac{n}{2}$, otherwise we will work on the polynomial $\Tilde{P}(X_1,\dots,X_n)=P(1-X_1,\dots,1-X_n)$.
    We shall show that there exists $P_{k}\in\FF_q[X_1,\dots,x_n]$ with $\deg(P_k)=k$ that vanishes at each $v\in\zone^n$ except when $\|u\|_{1} = k$. We define 
    $$
        P_{k}(X_1,\dots, X_n) := \sum_{I\in\binom{[n]}{k}}\left(\prod_{i\in I}X_i \right).
    $$
    Then $P_{k}\in\FF_q[X_1,\dots,X_n]$ satisfies the following:
    \begin{itemize}
        \item 
            $\deg(P_{k})=k$, 
        
        \item 
            $P_{k}(\Tilde{u})=0$, for all $\Tilde{u}\in\zone^n$ with $\| \Tilde{u} \|_{1} < k$, 
        
        \item 
            $P_{k}(u)=1$, for all $u\in\zone^n$ with $\| u \|_{1} = k$ and 

        \item 
            For all $u_r\in\zone^n$ with $\| u_{r} \|_{1} = k+r$, where $1\leq r\leq n-k$, we have
            $$
                P_{k}(u_r) \equiv \binom{k+r}{k}  \mod q.
            $$

%
    \end{itemize}

    We claim that for all $r\in [n-k]$, we have
    $$
        \binom{k+r}{k} \equiv 0 \mod q.
    $$
    By our assumption, $k+1$ is divisible by $q^m$. So in the expansion of $k+1$ with powers of $q$, coefficient of $q^i$ is $0$, for all $i<m$. So in the expansion of $k$ with powers of $q$, coefficient of $q^i$ is $q-1$, for all $i<m$. Now for each $r\in [n-k]\subseteq [q^m-1]$, there exists $j\in\{0,1,\dots,m-1\}$ such that the coefficient of $q^j$ in the expansion of $r$ with powers of $q$ is non-zero. So at least one carry is needed when adding $k$ with $r$ in base $q$. So by Theorem~\ref{kummer}, $q$ divides $\binom{k+r}{k}$, as required.
\end{proof}  

Barrington, Beigel, and Rudich~\cite{BBR94} showed that if $P(X_{1}, \dots, X_{n}) \in \mathbb{Z}_{m}[X_{1}, \dots, X_{n}]$ is a polynomial with minimum degree such that $P$ vanishes at $\mathcal{O}$ and does not vanishes at any other point of $\{0,1\}^{n}\setminus \{\mathcal{O}\}$ then $deg(P) = \Theta(n^{{1}/{r}})$, where
$r$ is the number of distinct primes dividing $m$.
Note that, Proposition~\ref{prop: tightness}, is a generalization of their result in the sense that the zero set of the polynomial is any $k$-th layer such that $k$ satisfies the assumptions and not necessarily $k=0$.

\color{black}
\section{Covering with multiplicities when $\mathbb{F} = \mathbb{R}$}


 Here we shall first study the following more generalized problem: For a given subset $S\subset \{0,1\}^n$ and a natural number $t$, what is the minimum degree of a polynomial with real coefficients that vanishes at each point of $S$ with multiplicity at least $t$ and does not vanish at least at one point of $\cQ^n\setminus S$?
 A polynomial $P(X_1,\dots,X_n)\in\RR[x_1,\dots,x_n]$ is said to have a \emph{zero of multiplicity $t$} at $a\in\RR^n$ if $P$ vanishes at $a$ and all derivatives of $P$ upto order $t-1$ also vanish at $a$.

Given a polynomial $P(X_{1}, \dots, X_{n}) \in \RR[X_{1}, \dots, X_{n}]$ we denote by $w_{t}(P) \in [n]$ the largest number such that for all $u_{1} \in \{0,1\}^{n}$, with $\| u_{1} \|_{1} < w_{t}(P)$, $P$ has a zero of multiplicity $t$ at $u_{1}$, and additionally, there exists a $v_{1} \in \{0,1\}^{n}$ with $\| v_{1} \|_{1} = w_{t}(P)$ such that $P(v_1)\neq 0$. Similarly, we denote by $W_{t}(P) \in [n]$ the smallest number such that for all $u_{2} \in \{0,1\}^{n}$, with $\| u_{2} \|_{1} > W_{t}(P)$, $P$ has a zero of multiplicity $t$ at $u_{2}$, and there exists a $v_{2} \in \{0,1\}^{n}$ with $\|v_{2}\|_{1} = W_{t}(P)$ such that $P(v_2)\neq 0$.

 We will use the following important generalization of Alon and Furedi's result (Theorem~\ref{th: AF}) by Sauermann and Wigderson. 
\begin{theorem}
[Sauermann and Wigderson~\cite{SW20}]\label{prop: Yuval, t,t-1}
Let $t\geq 2$, $n\geq 2t-3$ and $P(X_1,\dots,X_n)\in\RR[x_1,\dots,x_n]$ be a polynomial  having zeros of multiplicity at least $t$ at all points in $\zone^n\setminus\{(0,\dots,0)\}$ and $P(0, \dots, 0)\neq 0$. Then $\deg(P)\geq n+2t-3$ and the bound is tight. 
\end{theorem}
\begin{theorem}
    \label{th: mult_R}
    Let $t\in \NN$ and $S$ be a proper subset of $\zone^{n}$. Suppose $P\in\RR[X_1,\dots,X_n]$ be a non-zero polynomial that vanishes on $S$ with multiplicity at least $t$ and there exists $\xi\in \zone^{n} \setminus S$ with $P(\xi) \neq 0$. 
    If $\max\left\{w_{t}(P), n - W_{t}(P)\right\}\geq 2t-3$ then
    $$
        \deg(P)\geq \max \left\{ w_{t}(P), n - W_{t}(P) \right\}+2t-3.
    $$ 
\end{theorem}

\begin{proof}
    Without loss of generality we assume that $w\geq n-W$, otherwise, we work with the polynomial $\widetilde{P}(X_{1}, \dots, X_{n}) := P(1-X_{1}, \dots, 1-X_{n})$ in place of the polynomial $P$.
Let $u = (u_{1}, \dots, u_{n}) \in \zone^{n} \setminus S$ such that $\| u \|_{1} =w$ and $P(u)\neq 0$. Then, for all $v\in\zone^n$ with $\| v \|_{1} <w$, we have $P(v)=0$ with multiplicity at least $t$. Without loss of generality we may assume that $u_i=1$ if and only if $i\in[w]$. Consider the polynomial $Q(X_1,\dots,X_w) : = P(1-X_1,\dots, 1-X_w,0,\dots,0)$. Then $\deg(Q)\leq\deg(P)$ and $Q(0,\dots,0)=P(u)\neq 0$.

Now we claim that, $Q(X_{1}, \dots, X_{w})$ vanishes everywhere on $\zone^w\setminus\{(0,\dots,0)\}$ with multiplicity at least $t$. Take any $v\in\zone^w\setminus\{(0,\dots,0)\}$ and let $\Tilde{v}\in\zone^n$ such that $\Tilde{v}_i=1-v_i$, for all $i\in [w]$ and $\Tilde{v}_i=0$, for all $i>w$.
Then $\| \Tilde{v} \|_{1} <w$ and so $\Tilde{v}$ is a zero of $P$ with multiplicity at least $t$. Hence $Q(v)=P(\Tilde{v})=0$.
Next consider any differential operator $D=\frac{\partial^{d_1}}{\partial X_1^{d_1}}\dots\frac{\partial^{d_w}}{\partial X_w^{d_w}}$ such that $d_i\geq 0$, for all $i\in [w]$ and $\sum_{i=1}^w d_i< t$. We shall show that, $DQ(v)=0$.

First observe that, we can write $P(X_1,\dots,X_n)=P_1(X_1,\dots,X_w)+P_2(X_1,\dots,X_n)$, where each monomial of $P_2$ involves at least one variable from $X_{w+1},\dots, X_n$ with degree at least $1$. So $P_2$ vanishes at each \remove{$(X_1,\dots,X_w,\underbrace{0,\dots,0}_{(n-w)\text{ times}})\in\RR^n$.}$\Tilde{X}\in\RR^n$ such that $\Tilde{X}_i=0$, for all $i>w$.
Again $D$ involves none of the variables from $X_{w+1},\dots, X_n$, each monomial of $DP_2$ involves at least one variable from $X_{w+1},\dots, X_n$ with degree at least $1$ and so $DP_2$ also vanishes at each $\Tilde{X}\in\RR^n$ such that $\Tilde{X}_i=0$, for all $i>w$. So we can write $DP_1(X_1,\dots,X_w)=DP_1(X_1,\dots,X_w)+DP_2(X_1,\dots, X_w,0,\dots,0)=DP(X_1,\dots, X_w,0,\dots,0)$, for all $(X_1,\dots,X_w)\in\RR^w$.

Next observe that, we can write $Q(X_1,\dots,X_w)=P_1(1-X_1,\dots,1-X_w)$, as $P_2(1-X_1,\dots,1-X_w,0,\dots,0)=0$. So
$$
    DQ(X_1,\dots,X_w)=DP_1(1-X_1,\dots,1-X_w)
    =DP(1-X_1,\dots,1-X_w,0,\dots,0),
$$ 
for all $(X_1,\dots,X_w)\in\RR^w$. Hence $DQ(v)=DP(\Tilde{v})=0$, as $P$ vanishes at $\Tilde{v}$ with multiplicity $t$ implies that $DP(\Tilde{v})=0$. 
%
%
%
Therefore, using Theorem~\ref{prop: Yuval, t,t-1}, we get $\deg(Q)\geq w+2t-3$.
\end{proof}

\begin{remark}
 The above result is tight. There exists a polynomial $P\in\RR[X_1,\dots,X_n]$ with $\deg(P)=w+2t-3$ such that $w\geq 2t-3$ and $P(v)=0$ with multiplicity at least $t$ for all $v\in\zone^n$ with $\|v\|_{1} < w$ and for each $\Tilde{w}\in\{w,\dots,n\}$ there exists $u\in\zone^n$ with $\| u \|_{1} =\Tilde{w}$ such that $P(u)\neq 0$.
\end{remark}

\begin{proof}
By~\ref{SW23}, we know that there exists a polynomial $Q\in\RR[X_1,\dots, X_w]$ of degree $w+2t-3$ that vanishes everywhere on $\zone^w$ with multiplicity at least $t$ except at $(1,\dots,1)\in\zone^w$, where $Q$ does not vanish at all.

 Now define $P(X_1,\dots,X_n)=Q(X_1,\dots,X_w)$. Then for all $v\in\zone^n$ with 
 $\| v \|_{1} < w$, $P$ vanishes at $v$ and for each $\Tilde{w}\in\{w,\dots,n\}$, there exists $u_{\Tilde{w}}\in\zone^n$ such that $u_{\Tilde{w},i}=1$ if and only if $i\in [\Tilde{w}]$ and $P(u_{\Tilde{w}})=Q(\underbrace{1,\dots,1}_{w\text{ times}})\neq 0$. 

 Next consider any differential operator $D=\frac{\partial^{d_1}}{\partial X_1^{d_1}}\dots\frac{\partial^{d_n}}{\partial X_n^{d_n}}$ such that $d_i\geq 0$, for all $i\in [n]$ and $\sum_{i=1}^n d_i< t$. We shall show that, $DP(v)=0$, for all $v\in\zone^n$ with $\| v \|_{1}<w$. First observe that, if $d_i>0$, for some $i> w$, then $DP$ is identically zero. So now we assume that $d_i=0$, for all $i>w$ and take any $v\in\zone^n$ with $\| v \|_{1} < w$. Then $\Tilde{v}=(v_1,\dots,v_w)\in\zone^w\setminus\{(1,\dots,1)\}$ and so $\Tilde{v}$ is a zero of $Q$ with multiplicity at least $t$. Hence $DQ(\Tilde{v})=0$ and so $DP(v)=DQ(\Tilde{v})=0$, as required.  
\end{proof}

\bibliographystyle{alpha}

\bibliography{references}

\end{document}